\documentclass[12pt]{article}
\usepackage{amsmath,amsthm,amscd,amssymb}
\usepackage[mathscr]{eucal}
\usepackage{amssymb}
\usepackage{latexsym}
\usepackage{enumerate} 

\newtheorem{Thm}{Theorem}[section]
\newtheorem{Prop}[Thm]{Proposition}

\theoremstyle{definition}

\newtheorem{Rem}[Thm]{Remark}

\DeclareMathOperator{\re}{Re}

\DeclareMathOperator{\sgn}{sgn}
\DeclareMathOperator{\Tr}{Tr}

\DeclareMathOperator{\rank}{rank}
\DeclareMathOperator{\ord}{ord}

\DeclareMathOperator{\diag}{diag} 

\numberwithin{equation}{section}
\newcommand{\Q}{\mathbb{Q}}

\newcommand{\C}{\mathbb{C}}
\newcommand{\Z}{\mathbb{Z}}

\newcommand{\hh}{\mathbb{H}}

\newcommand{\bfon}{{\boldsymbol 1}} 

\newcommand{\h}{\mathbb{H}}

\begin{document}

\title{On Siegel--Eisenstein series of level $p$ and their $p$-adic properties}
\author{Siegfried B\"ocherer, Keiichi Gunji and Toshiyuki Kikuta}
\maketitle


\begin{abstract}
  We construct a Siegel--Eisenstein series of level $p$ with a quadratic character mod $p$ which is
  a $U(p)$-eigenfunction with eigenvalue $1$, 
and calculate its Fourier coefficients explicitly. 
We show that this Siegel--Eisenstein series is a $p$-adic Siegel--Eisenstein series, i.e., 
it is a $p$-adic limit of a sequence of Siegel--Eisenstein series of level $1$.  
We prove also that the Siegel--Eisenstein series with a nonquadratic character mod $p$
constructed by Takemori is also a $p$-adic Siegel--Eisenstein series.  
\end{abstract}

\section{Introduction}
$p$-Adic limits of Eisenstein series are formally defined by
a Fourier series, whose coefficients are uniform $p$-adic limits of
the Fourier coefficients for the sequence of Eisenstein series
in question. One may ask whether such a formal series actually defines
a classical modular form. In most works on this subject, it is an issue
how to ``guess'' a candidate for such a classical modular form.
In our papers \cite{Bo-Ki1,Bo-Ki3} the candidate
came up naturally as a linear combination of
genus theta series for quadratic forms of level $p$ (after some work!).
By Siegel's theorem this candidate can then be rewritten as a
Siegel--Eisenstein series of level $p$ and one can show that it automatically
has to be an eigenfunction for the Hecke operator $U(p)$ with eigenvalue $1$.

The first aim of this paper is to provide a direct
construction of  such a Siegel--Eisenstein series of level $p$ with a quadratic character mod $p$, 
which is $U(p)$-eigenfunction with eigenvalue $1$. 
This is accomplished by using the functional equation between 
two Eisenstein series corresponding to some cusps, presented by \cite{Gu}.
This method also allows us to give its Fourier coefficients completely explicitly.

The second aim of this paper is to show that a simple comparison of the
Fourier coefficients of $p$-adic limits of the Eisenstein series (of level $1$)
on one hand and
with those of the level $p$ Eisenstein series with $U(p)$-eigenvalue $1$ on the other hand
give the most direct approach to
the problem described above (``guessing a candidate'') for many cases, in fact in a surprisingly simple manner. 
We will later (Remark \ref{sec:comp}) comment on the strength and applicability of the
different methods.

\paragraph{Notation} 
Let $n$ be a positive integer,
$\hh_{n}$ the Siegel upper half space of degree $n$,
and $\Gamma _n$ the Siegel modular group of degree $n$. 
Let $N$ be a positive integer and $\Gamma _0^{(n)}(N)$ the congruence subgroup of $\Gamma _n$ defined as 
\begin{align*}
&\Gamma _0^{(n)}(N):=\left\{ \begin{pmatrix}A & B \\ C & D \end{pmatrix}\in \Gamma _n \: \Big| \: C\equiv 0_n \bmod{N} \right\}.
\end{align*}

For a positive integer $k$ and a Dirichlet character 
$\chi $ mod $N$, we denote by $M_k(\Gamma _0^{(n)}(N),\chi )$
the space of Siegel modular forms of weight $k$ with character $\chi$ for $\Gamma _0^{(n)}(N)$. 
When $\chi = {\boldsymbol 1}_N$ (trivial character mod $N$), 
we write simply $M_k(\Gamma _0^{(n)}(N))$ for $M_k(\Gamma _0^{(n)}(N),{\boldsymbol 1}_N)$.

Any $F \in M_k(\Gamma _0^{(n)}(N), \chi )$ has a Fourier expansion of the form
\[
F(Z)=\sum_{0\leq T\in\Lambda_n}a_F(T)q^T,\quad q^T:=e^{2\pi i {\rm Tr}(TZ)},
\quad Z\in\mathbb{H}_n,
\]
where
\[
\Lambda_n
:=\{ T=(t_{ij})\in {\rm Sym}_n(\mathbb{Q})\;|\; t_{ii},\;2t_{ij}\in\mathbb{Z}\}.
\]

For the Riemann zeta function $\zeta(s)$ or the Dirichlet $L$-function $L(s,\chi)$ with character $\chi$, we denote by $\zeta^{(p)}(s)$ or $L^{(p)}(s,\chi)$ the functions removing the Euler $p$-factors for a prime $p$, that is 
\[ \zeta^{(p)}(s) = (1-p^{-s}) \zeta(s), \quad L^{(p)}(s, \chi) = (1-\chi(p) p^{-s})L(s,\chi). \]
For Dirichlet characters $\psi$ and $\chi$, we denote simply by $\psi \chi$ the primitive character arising from $\psi \cdot \chi$.



\section{Construction of Siegel--Eisenstein series}
\label{sec:Eisen}

In this section we construct a holomorphic Siegel--Eisenstein series of
level $\Gamma_0(p)$, which is a $U(p)$-eigenfunction with eigenvalue $1$.
We quote some results from \cite{Gu}.

\subsection{Siegel--Eisenstein series and the  $U(p)$-operator}
Let $p$ be an odd prime number. We denote the non-trivial quadratic character modulo $p$ by $\chi_p$. Let $\psi = {\boldsymbol 1}_p$ or $\chi_p$. For an integer $k$ such that $\psi(-1) = (-1)^k$ we define the Siegel--Eisenstein series of degree $n$, level $p$ and weight $k$ with character $\psi$ attached to cusp $w_n = \begin{pmatrix} 0 & -1_n \\ 1_n & 0 \end{pmatrix}$ by
\[
E^{n,(n)}_{k,\psi}(Z,s)
 = (\det Y)^{s-k/2} \sum_{\gamma \in \Gamma_\infty \backslash \Gamma_\infty w_n \Gamma_0^{(n)}(p)} \psi(\det C)\det(CZ+D)^{-k} |\det(CZ+D) |^{k-2s}.
\]
for $Z \in \h_n$ and $s \in \C$. Here we write $\gamma = \begin{pmatrix} A & B \\ C & D \end{pmatrix}$ as usual. The right hand side converge for $\re(2s)> n+1$ and it continues meromorphically to whole $s$-plane.  If $k > n+1$, we have $E^{n,(n)}_{k,\psi}(Z,k/2) \in M_k(\Gamma_0^{(n)}(p), \psi)$. 

Let $\mathcal{E}_{k,s}(\Gamma^{(n)}_0(p),\psi)$ be the space of Siegel--Eisenstein series of weight $k$, with complex parameter $s$ (cf.\ \cite[\S\S 2.1]{Gu}). Then the Hecke operator $U(p)$ acts on $\mathcal{E}_{k,s}(\Gamma^{(n)}_0(p),\psi)$ as follows. For $f = \sum_{T \in \Lambda_n} b(T,Y,s)e^{2 \pi i \mathrm{tr}(TX)} \in \mathcal{E}_{k,s}(\Gamma_0^{(n)}(p),\psi)$, we have
\[ U(p)f = \sum_{T \in \Lambda_n} b(pT,p^{-1}Y,s) e^{2 \pi i \mathrm{tr}(TX)}. \]
Then by \cite[Proposition 2.10]{Gu}, $E^{n,(n)}_{k,\psi}(Z,s)$ is a $U(p)$-eigenfunction with the eigenvalue $p^{n(k/2-s)+2sn-n(n+1)/2}$. We know the existence and the uniqueness up to scalar of such a function by \cite[Corollary 2.5]{Gu}.

Let $E^{n,(0)}_{k,\psi}(Z,s) \in \mathcal{E}_{k,s}(\Gamma^{(n)}_0(p),\psi)$ be the $U(p)$-eigenfunction with eigenvalue $p^{n(k-s/2)}$, normalized so that the constant term of the Fourier expansion is $\det(Y)^{s/2-k}$. Then there is a functional equation between $E^{n,(0)}_{k,\psi}(Z,s)$ and $E^{n,(n)}_{k,\psi}(Z,(n+1)/2-s)$ as follows.

We put
\[ \delta = \begin{cases} 0 & \text{if $k$ is even,} \\ 1 & \text{if $k$ is odd,} \end{cases} \quad \varepsilon_p = \begin{cases} 1 & \text{if $p \equiv 1 \bmod 4$,} \\ i &  \text{if $p \equiv 3 \bmod 4$.} \end{cases} \]
We define the completed Riemann zeta function and the Dirichlet $L$-function by
\[ \xi(s) = \pi^{-s/2} \Gamma \left(\frac{s}{2} \right) \zeta(s), \quad \xi(s, \chi_p) = \left(\frac{p}{\pi} \right)^{(s+\delta)/2} \Gamma \left(\frac{s+\delta}{2} \right) L(s, \chi_p). \]
Then $\xi(1-s) = \xi(s)$ and $\xi(1-s, \chi_p) = \xi(s,\chi_p)$ holds. By abuse of language, we understand $\xi(s,\psi) = \xi(s)$ for $\psi = \bfon_p$. We set $\Gamma_m(s) = \pi^{m(m-1)/4} \prod_{i=0}^{m-1} \Gamma(s-i/2$).

Now we put
\begin{align*}  \mathbb{E}_{k,\psi}^{n,(0)}(Z,s) & = \frac{\Gamma_n\left(s + \dfrac{k}{2} \right)}{\Gamma_n \left( s+ \dfrac{\delta}{2} \right) } \xi(2s,\psi) \prod_{j=1}^{[n/2]} \xi(4s-2j) E_{k,\psi}^{n,(0)}(Z,s), \\
\mathbb{E}_{k,\psi}^{n,(n)}(Z,s) & = \gamma_n(s,\psi) \frac{\Gamma_n\left(s + \dfrac{k}{2} \right)}{\Gamma_n \left( s + \dfrac{\delta}{2} \right)} \xi(2s,\psi) \prod_{j=1}^{[n/2]} \xi(4s-2j) E_{k,\psi}^{n,(n)}(Z,s),
\end{align*}
here
\begin{align*} 
\gamma_n(s, \bfon_p) & = \frac{1-p^{-2s}}{1-p^{n-2s}} \prod_{i=1}^{[n/2]} \frac{1-p^{2i-4s}}{1-p^{2n+1-2i-4s}}, \\
\gamma_n(s,\chi_p) & = (\varepsilon_p p^{-1/2})^n \prod_{i=1}^{[n/2]} \frac{1-p^{2i-4s}}{1-p^{2n+1-2i-4s}}. 
\end{align*}
Then by \cite[Theorem 4.3, Theorem 4.8]{Gu} the following functional equation holds.
\[ \mathbb{E}^{n,(0)}_{k,\psi}(Z,s) = \mathbb{E}^{n,(n)}_{k,\psi} \left( Z, \frac{n+1}{2}- s \right). \]
Thus we can write
\begin{align} \label{eq:En(0)}
E^{n,(0)}_{k,\psi}(Z,s) & = \gamma_n \Bigl(\frac{n+1}{2}-s,\psi \Bigr) \frac{\Gamma_n \Bigl(s + \dfrac{\delta}{2} \Bigr) \Gamma_n \Bigl( \dfrac{n+1}{2} -s + \dfrac{k}{2} \Bigr)}{\Gamma_n\Bigl(s + \dfrac{k}{2} \Bigr) \Gamma_n \Bigl( \dfrac{n+1+\delta}{2}-s \Bigr)}  \\
& \times \frac{\xi(n+1-2s, \psi)}{\xi(1-2s, \psi )} \prod_{j=1}^{[n/2]} \frac{\xi(2n+2-4s-2j)}{\xi(1-4s+2j)} E^{n,(n)}_{k,\psi} \Bigl(Z, \frac{n+1}{2}-s\Bigr), \notag
\end{align}
here we use $\xi(2s,\psi) = \xi(1-2s, \psi)$ and $\xi(4s-2j) = \xi(1-4s+2j)$. 

In this section we shall show the following.

\begin{Thm} \label{thm:low weight Eisen}
For any positive integer $k$ such that $\psi(-1) = (-1)^k$, $E^{n,(0)}_{k,\psi}(Z,k/2)$ is a non-zero holomorphic Siegel modular form of level $p$, weight $k$ with character $\psi$. It is a $U(p)$-eigenfunction with eigenvalue $1$.
\end{Thm}


In order to prove the theorem, we shall study the Fourier expansion of $E^{n,(0)}_{k,\psi}(Z,s)$. For $T \in \Lambda_n$ and a prime number $q$, we define the local Siegel series by
\[ S^n_q(T,q^{-s}) = \sum_{R \in {\rm Sym}_n(\Q_q/\Z_q)} \mu_q(R)^{-s} e^{2 \pi i \Tr(TR)}, \]
where $\mu_q(R) = q^{\ord_q \mu(R)}$ with $\mu(R)$ the product of denominators of all elementary divisors of $R$. It is known that $S^n_q(T,q^{-s})$ is a rational function in $q^{-s}$.

Using Siegel series, the Fourier expansion of  $E^{n,(n)}_{k,\psi}$ is given by
\begin{align*}   & \quad E^{n,(n)}_{k,\psi}\Bigl(Z,\frac{n+1}{2}-s \Bigr) \\
&  = \det(Y)^{(n+1-k)/2 -s}  \sum_{T \in \Lambda_n} \biggl\{  \Xi_n(Y, T;  s_1, s_2 ) \prod_{q \ne p} S^n_{q}(T, \psi(q)q^{2s-n-1}) \biggr\} e^{2 \pi i \Tr(TX)},
\end{align*}
with $s_1 = (n+1)/2-s+k/2$ and $s_2 = (n+1)/2-s-k/2$. Here, $\Xi_n(Y,T;\alpha, \beta)$ is the confluent hypergeometric function investigated by Shimura (\cite{Shi1}). We use the notation $\Xi$ instead of $\xi$, in order to avoid confusions with completed zeta functions. 

First we study $\Xi_n(Y,T;s_1,s_2)$.
For $T \in \Lambda_n$ we denote the sign of $T$ by $(a,b)$, when the number of the positive  (resp.\ negative) eigenvalues of $T$ is $a$ (resp.\ $b$). By \cite[(4.34K)]{Shi1} we have a following formula. Let $T \in \Lambda_n$ with $\rank T = r$. If the sign of $T$ is $(r-t,t)$ ($0 \le t \le r \le n$), we have
\begin{equation} \label{eq:Xi_n}
\begin{split}
& \Xi_n (Y,T; s_1, s_2 )= i^{-nk} 2^u \pi^v \frac{\Gamma_{n-r}\Bigl(  \dfrac{n+1}{2}-2s \Bigr)}{ \Gamma_{n-t}\Bigl( \dfrac{n+1}{2}-s+\dfrac{k}{2} \Bigr) \Gamma_{n-r+t}\Bigl( \dfrac{n+1}{2}-s-\dfrac{k}{2} \Bigr)} \\ 
& \ \  \times \det(Y)^{2s-(n+1)/2} d_+(TY)^{-s+k/2 + t/4} d_-(TY)^{-s-k/2+(r-t)/4} \omega_n(2 \pi Y, T; s_1, s_2).
\end{split}
\end{equation}
Here
\begin{align*}
u & = \frac{r(n+1)}{2} + 2(n-r)s + k(r-2t) + \frac{(r-t)t}{2}-\frac{n(n-1)}{2}, \\
v & = \frac{(k-2s)r}{2} - kt + \frac{n(n+1)}{2} - \frac{r(n-r)}{2}-\frac{t(r-t)}{2},
\end{align*}
$d_+(TY)$ (resp.\ $d_-(TY)$) is the product of the absolute values of all positive (resp.\ negative) eigenvalues of $TY$ and $\omega_n(2 \pi Y, T, \alpha, \beta)$ is a certain entire function in $(\alpha,\beta) \in \C^2$. If  $T$ is positive semi-definite, i.e.\ $t=0$, we have by \cite[(4.35K)]{Shi1},
\begin{equation} \label{eq:omega_n special}
\omega_n \left( 2 \pi Y, T; \frac{n+1}{2}, \beta \right) = 2^{-r(n+1)/2} \pi^{r(n-r)/2} e^{-2 \pi \Tr(TY)}. 
\end{equation}

Next we investigate the Siegel series. Assume that $T \in \Lambda_n$ is non-degenerate. If $n$ is even, we denote the quadratic character associated with the quadratic extension $\Q(\sqrt{(-1)^{n/2} \det(2T)})/\Q$  by $\chi_T$. Then it is known (e.g.,\ \cite[Proposition 3.6]{Shi2}) that there exists a certain polynomial $F^n_q(T,X)$ in $X$ such that 
\begin{align*} 
 S^n_q(T,q^{-s})  =   \begin{cases} \displaystyle (1-\chi_T(q)q^{n/2-s})^{-1} (1-q^{-s}) \left( \prod_{i=1}^{n/2}(1-q^{2i-2s}) \right)F^n_q(T,q^{-s}) &  \text{$n$ is even,} \\ 
\displaystyle (1-q^{-s}) \left( \prod_{i=1}^{(n-1)/2}(1-q^{2i-2s}) \right)  F_q^{n}(T,q^{-s}) &  \text{$n$ is odd.}\end{cases}	
\end{align*}
Note that $F^n_q(T,q^{-s}) = 1$ unless $q \mid \det(2T)$. An explicit formula of $F^n_q(T,q^{-s})$ is given by Katsurada (\cite[Theorem 4.3]{Kat}).

For $T \in \Lambda_n$ with lower rank, we have the following.

\begin{Prop}[{\cite[Theorem 3.2, (3.4)]{Shi2}}] \label{prop:Siegel series}
Let $T_1 \in \Lambda_r$ with $\det T_1 \ne 0$ and $T = \diag(T_1,0) \in \Lambda_n$.  If $r$ is even then
\begin{align*} S_q^{n} ( T, q^{-s}) =   \frac{ \displaystyle (1-q^{-s}) \prod_{j=1}^{[n/2]}(1-q^{2j-2s})}{ \displaystyle (1-\chi_{T_1}(q)q^{n-r/2-s}) \prod_{j=1}^{[(n-r)/2]}(1-q^{2n-2j-r+1-2s}) } F_q^r(T_1, q^{n-r-s}). 
\end{align*}
If $r$ is odd then
\begin{align*} S_q^{n} ( T, q^{-s}) =   \frac{ \displaystyle (1-q^{-s}) \prod_{j=1}^{[n/2]}(1-q^{2j-2s})}{ \displaystyle  \prod_{j=1}^{[(n-r+1)/2]}(1-q^{2n-2j-r+2-2s}) } F_q^r(T_1, q^{n-r-s}). 
\end{align*}

\end{Prop}

\subsection{The case of trivial character}

In this subsection we treat the case of $\psi = \bfon_p$. In that case the weight $k$ is even. Let
\[ E^{n,(0)}_{k,\bfon_p}(Z,s) = \sum_{T \in \Lambda_n} C_{\bfon_p}(T;Y,s) e^{2 \pi i \Tr(TX)} \]
be the Fourier expansion of $E^{n,(0)}_{k,\bfon_p}$.

First assume that $T$ is positive semi-definite; the sign of $T$ is $(r,0)$ ($0 \le r \le n$). Let $T = \diag(T_1,0)$ with $T_1 \in \Lambda_r$ and $T_1 > 0$.
We can write $C_{\bfon_p}(T;Y,s)$ explicitly using (\ref{eq:En(0)}), (\ref{eq:Xi_n}) and Proposition \ref{prop:Siegel series}. The term in (\ref{eq:En(0)}) is rewritten as
\begin{align*}
& \quad \ \  \frac{\xi(n+1-2s)}{\xi(1-2s)} \prod_{j=1}^{[n/2]} \frac{\xi(2n+2-4s-2j)}{\xi(1-4s+2j)}  \\
& = \pi^{-n(n+1)/4} \frac{\Gamma \Bigl(\dfrac{n+1}{2}-s\Bigr)}{\Gamma \Bigl( \dfrac{1}{2}-s \Bigr)} \prod_{j=1}^{[n/2]} \frac{\Gamma(n+1-2s-j)}{\Gamma\Bigl(\dfrac{1}{2} -2s + j\Bigr)} \\
& \quad \ \ \times  \frac{\zeta(n+1-2s)}{\zeta(1-2s)} \prod_{j=1}^{[n/2]} \frac{\zeta(2n+2-4s-2j)}{\zeta(1-4s+2j)}.
\end{align*}
Together with (\ref{eq:omega_n special}) we can write
\[ C_{\bfon_p}(T;Y,k/2) = \lim_{s \to k/2} G_{n,r}(s) Z_n(\bfon_p,T,s) e^{-2 \pi \Tr(TY)}. \]
Here, $G_{n,r}(s)$ is the term corresponding to Gamma functions, defined by
\begin{align} \label{eq:G_nr_0}
  G_{n,r}(s) & = (-1)^{nk/2} 2^{2ns-n(n-1)/2} \pi^{n(n+1)/4}  \notag \\
& \times   \frac{\Gamma_n(s) \Gamma_{n-r} \Bigl(\dfrac{n+1}{2}-2s \Bigr)}{\Gamma_n \Bigl(s+\dfrac{k}{2} \Bigr) \Gamma_n\Bigl(\dfrac{n}{2}-s \Bigr) \Gamma_{n-r} \Bigl(\dfrac{n+1}{2}-s-\dfrac{k}{2} \Bigr)} \\
& \times  \prod_{j=1}^{[n/2]} \frac{\Gamma(n+1-2s-j)}{\Gamma\Bigl(\dfrac{1}{2} -2s + j\Bigr)}, \notag
\end{align}
where we use the relation $\dfrac{\displaystyle \Gamma \Bigl( \frac{n+1}{2}-s \Bigr)}{\displaystyle \Gamma \Bigl( \frac{1}{2}-s \Bigr) \Gamma_{n} \Bigl(\frac{n+1}{2}-s \Bigl)} = \dfrac{1}{\Gamma_n\Bigl(\dfrac{n}{2}-s \Bigr)}$.

The term $Z_n(\bfon_p,T,s)$ is the product of zeta functions. 
If $r$ is even, we have
\[Z_n(\bfon_p,T,s) = \frac{L^{(p)}(1+r/2-2s, \chi_{T_1})}{\displaystyle \zeta^{(p)}(1-2s) \prod_{j=1}^{r/2} \zeta^{(p)}(1-4s+2j)} \prod_{q \ne p} F^r_q(T_1,q^{2s-r-1}), \] 
and if $r$ is odd
\[Z_n(\bfon_p,T,s) = \frac{1}{\displaystyle \zeta^{(p)}(1-2s) \prod_{j=1}^{(r-1)/2} \zeta^{(p)}(1-4s+2j)} \prod_{q \ne p} F^r_q(T_1,q^{2s-r-1}). \] 

We study $G_{n,r}(s)$ more precisely. If $n$ is even, then 
\[ \prod_{j=1}^{[n/2]} \frac{\Gamma(n+1-2s-j)}{\Gamma\Bigl(\dfrac{1}{2}-2s + j \Bigr)} = \prod_{j=0}^{n/2-1} \frac{\Gamma(n-2s-j)}{\Gamma\Bigl(\dfrac{n+1}{2}-2s-j \Bigr)} = \prod_{j=0}^{n-1}  \frac{\Gamma(n-2s-j)}{\Gamma \Bigl( \dfrac{n+1}{2}-2s - \dfrac{j}{2} \Bigr)}, \]
and if $n$ is odd, then
\[ \prod_{j=1}^{[n/2]} \frac{\Gamma(n+1-2s-j)}{\Gamma\Bigl(\dfrac{1}{2}-2s + j \Bigr)} = \prod_{j=0}^{(n-3)/2} \frac{\Gamma(n-2s-j)}{\Gamma\Bigl(\dfrac{n}{2}-2s-j \Bigr)} = \prod_{j=0}^{n-1}  \frac{\Gamma(n-2s-j)}{\Gamma \Bigl( \dfrac{n+1}{2}-2s - \dfrac{j}{2} \Bigr)} \]
also holds. 
By Legendre's duplication formula, we have
\[  \prod_{j=0}^{n-1} \Gamma(n-2s-j) 
 = \pi^{-n^2/2} 2^{n(n-1)/2-2ns} \Gamma_n \Bigl( \frac{n+1}{2}-s \Bigr) \Gamma_n \Bigl( \frac{n}{2} -s \Bigr).\]
Together with the functional equation shown by Mizumoto (\cite[Lemma 6.1]{Miz})
\[ \frac{\Gamma_n(s)}{\Gamma_n\Bigl(s+ \dfrac{k}{2} \Bigr)} = (-1)^{nk/2} \frac{\Gamma_n\Bigl( \dfrac{n+1}{2}-s-\dfrac{k}{2} \Bigr)}{\Gamma_n \Bigl( \dfrac{n+1}{2}-s \Bigr)}, \]
we have
\begin{align} \label{eq:G_nr}
G_{n,r}(s)  & =  \frac{\Gamma_n\Bigl(\dfrac{n+1}{2}-s-\dfrac{k}{2} \Bigr) \Gamma_{n-r} \Bigl(\dfrac{n+1}{2}-2s \Bigr)}{\Gamma_{n-r}\Bigl(\dfrac{n+1}{2}-s-\dfrac{k}{2} \Bigr) \Gamma_{n} \Bigl(\dfrac{n+1}{2}-2s \Bigr)}  
\notag \\ 
& = \prod_{j=0}^{r-1} \frac{\Gamma\Bigl(\dfrac{r+1}{2}-s-\dfrac{k}{2}-\dfrac{j}{2} \Bigr)}{\Gamma\Bigl( \dfrac{r+1}{2}-2s - \dfrac{j}{2} \Bigr)}. 
\end{align}
Both the numerator and the denominator have  simple poles at $s = k/2$ for $j$ such that $r+1-2k-j$ is an even negative integer.
Let $l = \sharp \{\max(0, r+1-2k ) \le j \le r-1 \mid j \equiv r+1 \bmod 2 \}$. Then we have   $\lim_{s \to k/2} G_{n,r}(s) = 2^l$. In particular when $r \le 2k-1$,
\[ \lim_{s \to k/2} G_{n,r}(s) = \begin{cases} 2^{r/2} & \text{if $r$ is even, } \\ 2^{(r+1)/2} & \text{if $r$ is odd.} \end{cases} \]

Next we study the term $Z_n(\bfon_p,T,s)$. The numerator of $Z_n(\bfon_p,T,s)$ has a pole at $s = k/2$ only when $r = 2k$ and $\chi_{T_1} = \bfon$ (trivial character mod $1$).
On the other hand, thanks to $\zeta^{(p)}(1)$, the denominator has a pole at $s = k/2$  if and only if $r \ge 2k$. As a consequence, $Z_n(\bfon_p,T,s)$ vanishes at $s = k/2$ unless $r \le 2k$, and if $r = 2k$ then only the term $T = \diag(T_1,0)$ with $\chi_{T_1} = \bfon$ survives.

Next we consider the case of non positive semi-definite. Let $T = \diag( T_1,0)$ with $\det T_1 \ne 0$ and $\sgn T_1 = (r-t, t)$ for $r \ge t \ge 1$. We shall show that in that case the Fourier coefficient at $T$ vanishes. 

Note that $Z_n(\bfon_p,T,s)$ is same as the positive semi-definite case. On the other hand for $G_{n,r}(s)$ the Gamma functions are changed. By (\ref{eq:Xi_n}), ignoring the non-zero terms, $G_{n,r}(s)$ is given by $(\ref{eq:G_nr}) \times R(s)$ with
\begin{align*} 
R(s) & = \frac{\Gamma_{n} \Bigl(\dfrac{n+1}{2}-s + \dfrac{k}{2} \Bigr) \Gamma_{n-r} \Bigl(\dfrac{n+1}{2}-s-\dfrac{k}{2} \Bigr)}{\Gamma_{n-t} \Bigl(\dfrac{n+1}{2}-s + \dfrac{k}{2} \Bigr) \Gamma_{n-r+t} \Bigl(\dfrac{n+1}{2}-s-\dfrac{k}{2} \Bigr)} \\ 
&  = \frac{\displaystyle \pi^*   \prod_{i=0}^{t-1} \Gamma \Bigl( \frac{t+1}{2}-s+\frac{k}{2} -\frac{i}{2} \Bigr)}{\displaystyle \prod_{i=0}^{t-1} \Gamma \Bigl( \frac{r+1}{2}-s-\frac{k}{2} - \frac{i}{2} \Bigr)}.
\end{align*}
The numerator is finite at $s = k/2$. Thus in order to show the vanishing of Fourier coefficients, it suffices to consider the case $r \le 2k$ since otherwise $Z_n(\bfon_p,T,k/2) = 0$ as we have seen above.
Then thanks to the denominator, it vanishes at $s = k/2$ unless $t=1$ and $r$ is even. 

Let us consider the remaining case,  $t = 1$ and $r$ is even. We see that the term $L^{(p)}(1+r/2-2s, \chi_{T_1})$ in $Z_n(\bfon_p,T,s)$ vanishes at $s = k/2$ for $r \le 2k-2$. Indeed, then $r/2+1-k$ is a non-positive integer. Since $\sgn T_1 = (r-1,1)$ we have $\chi_{T_1}(-1) = (-1)^{r/2+1}$, thus $L^{(p)}(r/2+1-k, \chi_{T_1})= 0$, for we have assumed $k$ is even. If $r = 2k$, then as we have seen above, $Z_n(\bfon_p,T,k/2)$ survives only for the $T$ such that $\chi_{T_1}= \bfon$, but it does not happen since $\chi_{T_1}(-1) = -1$. As a consequence we have shown that $C_{\bfon_p}(T,Y,k/2) = 0$ for non positive semi-definite $T$.

Now we have proved Theorem \ref{thm:low weight Eisen}. More precisely we can write down the Fourier coefficients as follows.

\begin{Thm} \label{thm: F-C trivial case}
Let $T = \diag(T_1,0) \in \Lambda_n$ with $T_1 \in \Lambda_r$ positive definite. Then the Fourier coefficient $a(T)$ of $E^{n,(0)}_{k,\bfon_p}(Z,k/2)$ at $T$ is given as follows.
\begin{enumerate}[$(1)$]
\item Assume that $k \ge \dfrac{n+1}{2}$. If $r$ is even then
\[ a(T) = \frac{ 2^{r/2} \, L^{(p)}(r/2+1-k,\chi_{T_1})}{ \displaystyle \zeta^{(p)}(1-k) \prod_{j=1}^{r/2} \zeta^{(p)}(1-2k+2j)} \, \prod_{q \ne p} F_q^r(T_1, q^{k-r-1}). \]
If $r$ is odd then
\[  a(T) = \frac{ 2^{(r+1)/2} }{ \displaystyle \zeta^{(p)}(1-k) \prod_{j=1}^{(r-1)/2} \zeta^{(p)}(1-2k+2j)} \, \prod_{q \ne p} F_q^{r}(T_1, q^{k-r-1}). \]

\item Assume that $k \le \dfrac{n}{2}$. Then $a(T) = 0$ unless $r \le 2k$. If $r<2k$ then $a(T)$ is the
  same as above. For the case $r = 2k$,  $a(T) \ne 0$ only if $\chi_{T_1} = \bfon$ and in that case,
\[ a(T) = \frac{ 2^{k} }{ \displaystyle \zeta^{(p)}(1-k) \prod_{j=1}^{k-1} \zeta^{(p)}(1-2k+2j)} \, \prod_{q \ne p} F_q^r(T_1, q^{-k-1}). \]
\end{enumerate}
\end{Thm}

\subsection{The case of quadratic character}

In this subsection we treat the case $\psi = \chi_p$ the quadratic character modulo $p$. The necessary calculations are quite similar to the case of trivial character. We write the Fourier expansion
\[ E^{n,(0)}_{k,\chi_p}(Z,s) = \sum_{T \in \Lambda_n} C_{\chi_p}(T;Y,s) e^{2 \pi i \Tr(TX)}. \]
If $T = \diag(T_1,0)$ is positive semi-definite, then we have
\[ C_{\chi_p}(T;Y,k/2) = \lim_{s \to k/2} G^{\chi_p}_{n,r}(s) Z_n(\chi_p,T,s) e^{-2 \pi \Tr(TY)} \]
with 
\begin{align} \label{eq:G_nr^chi_0}
G^{\chi_p}_{n,r}(s) & = \varepsilon_p^{n} \, i^{-nk} 2^{2ns-n(n-1)/2} \pi^{n(n+1)/4}  \notag \\
& \times   \frac{\Gamma_n\Bigl(s + \dfrac{\delta}{2} \Bigr) \Gamma_{n-r} \Bigl(\dfrac{n+1}{2}-2s \Bigr)}{\Gamma_n \Bigl(s+\dfrac{k}{2} \Bigr) \Gamma_n\Bigl(\dfrac{n+\delta}{2}-s \Bigr) \Gamma_{n-r} \Bigl(\dfrac{n+1}{2}-s-\dfrac{k}{2} \Bigr)} \\ 
& \times \prod_{j=1}^{[n/2]} \frac{\Gamma(n+1-2s-j)}{\Gamma\Bigl(\dfrac{1}{2} -2s + j\Bigr)} \notag
\end{align}
and $Z_n(\chi_p,T,s)$ defined as
\[Z_n(\chi_p,T,s) = \frac{L^{(p)}(1+r/2-2s, \chi_{T_1}\chi_p)}{\displaystyle L(1-2s,\chi_p) \prod_{j=1}^{r/2} \zeta^{(p)}(1-4s+2j)} \prod_{q \ne p} F^r_q(T_1,\chi_p(q)q^{2s-r-1}), \] 
or
\[Z_n(\chi_p,T,s) = \frac{1}{\displaystyle L(1-2s,\chi_p) \prod_{j=1}^{(r-1)/2} \zeta^{(p)}(1-4s+2j)} \prod_{q \ne p} F^r_q(T_1,\chi_p(q)q^{2s-r-1}) \] 
according as $r$ is even or odd. We recall that $\chi_{T_1} \chi_p$ stands for the primitive character arising from $\chi_{T_1} \cdot \chi_p$.
By the similar way to show (\ref{eq:G_nr}), we can prove the equality
\begin{equation} \label{eq:G_nr^chi}
 G_{n,r}^{\chi_p}(s) =  \prod_{j=0}^{r-1} \frac{\Gamma\Bigl(\dfrac{r+1}{2}-s-\dfrac{k}{2}-\dfrac{j}{2} \Bigr)}{\Gamma\Bigl( \dfrac{r+1}{2}-2s - \dfrac{j}{2} \Bigr)}.
\end{equation}
Indeed, if $k$ is even, then (\ref{eq:G_nr^chi_0}) is completely same as (\ref{eq:G_nr_0}). 
If $k$ is odd, then we have $p \equiv 3 \bmod 4$ and $\varepsilon_p^k \, i^{-n(k+1)} = (-1)^{n(k-1)/2}$ in (\ref{eq:G_nr^chi_0}). Moreover by \cite[Lemma 6.1]{Miz} we have the functional equation
\[ \frac{\Gamma_n \Bigl(s + \dfrac{1}{2} \Bigr)}{\Gamma_n\Bigl(s+ \dfrac{k}{2} \Bigr)} = (-1)^{n(k-1)/2} \frac{\Gamma_n\Bigl( \dfrac{n+1}{2}-s-\dfrac{k}{2} \Bigr)}{\Gamma_n \Bigl( \dfrac{n}{2}-s \Bigr)}. \]
Using it  one can show the equation (\ref{eq:G_nr^chi}).

Therefore we have for $r \le 2k-1$
\[ \lim_{s \to k/2} G^{\chi_p}_{n,r}(s) = \begin{cases} 2^{r/2} & \text{if $r$ is even, } \\ 2^{(r+1)/2} & \text{if $r$ is odd.} \end{cases} \]
Moreover if $\sgn T = (r-t,t)$ with $1 \le t \le r \le n$, then $\lim_{s \to k/2} G^{\chi_p}_{n,r}(s) = 0$ unless $t=1$ and $r$ is even.

The term $Z_{n}(\chi_p,T,s)$ is also similar to the case of the trivial character. It has zero at $s = k/2$ for $r \ge 2k+1$ and if $r = 2k$, it survives only for $T = (T_1,0)$ with $T_1 \in \Lambda_{2k}$ that satisfies $\chi_{T_1} = \chi_p$. In the case that $\sgn T_1 = (r-1,1)$ with even $r$, then $\chi_{T_1} \chi_p(-1) = (-1)^{r/2-k+1}$, thus $L^{(p)}(1+r/2-k, \chi_{T_1}\chi_p) = 0$ for $r \le 2k-1$, and $\chi_{T_1} = \chi_p$ does not happen if $r = 2k$. Thus we have shown that $C_{\chi_p}(T;Y,k/2) = 0$ for non positive semi-definite $T$.

Now we have proved Theorem \ref{thm:low weight Eisen}. The Fourier coefficients are explicitly given as follows. We remark that in the case of $k>n+1$, the same formula was given by Takemori (\cite[Theorem 2.3]{Ta}), who  used the functional equation of the Siegel series.

\begin{Thm} \label{thm: F-C nontrivial case}
Let $T = \diag(T_1,0) \in \Lambda_n$ with $T_1 \in \Lambda_r$ positive definite. Then the Fourier coefficient $a(T)$ of $E^{n,(0)}_{k,\chi_p}(Z,k/2)$ at $T$ is given as follows.
\begin{enumerate}[$(1)$]
\item Assume that $k \ge \dfrac{n+1}{2}$. If $r$ is even then
\[ a(T) = \frac{ 2^{r/2} \, L^{(p)}(r/2+1-k,\chi_{T_1}\chi_p)}{ \displaystyle L(1-k,\chi_p) \prod_{j=1}^{r/2} \zeta^{(p)}(1-2k+2j)} \, \prod_{q \ne p} F_q^r(T_1, \chi_p(q)q^{k-r-1}). \]
If $r$ is odd then
\[  a(T) = \frac{ 2^{(r+1)/2} }{ \displaystyle L(1-k,\chi_p) \prod_{j=1}^{(r-1)/2} \zeta^{(p)}(1-2k+2j)} \, \prod_{q \ne p} F_q^{r}(T_1, \chi_p(q)q^{k-r-1}). \]

\item Assume that $k \le \dfrac{n}{2}$. Then $a(T) = 0$ unless $r \le 2k$. If $r<2k$ then $a(T)$ is the same as above. For the case $r = 2k$,  $a(T) \ne 0$ only if $\chi_{T_1} = \chi_p$ and in that case,
\[ a(T) = \frac{ 2^{k} }{ \displaystyle L(1-k,\chi_p) \prod_{j=1}^{k-1} \zeta^{(p)}(1-2k+2j)} \, \prod_{q \ne p} F_q^r(T_1, \chi_p(q)q^{-k-1}). \]
\end{enumerate}
\end{Thm}

\section{$p$-Adic limits of the Siegel--Eisenstein series and their Fourier coefficients}
Let $p$ be an odd prime and $v_p$ the additive valuation on the $p$-adic field $\Q_p$ normalized such that $v_p(p)=1$. 
For a formal power series $F$ of the form $F=\sum _{T\in \Lambda _{n}}a_{F}(T)q^T$ with $a_{F}(T)\in \Q_p$,  
we define
\begin{align*}
&v_p(F):=\inf \{v_p(a_{F}(T))\; |\; T\in \Lambda _n\}.  
\end{align*}
For a sequence of formal power series 
$F_{m}=\sum _{T\in \Lambda _{n}}a_{F_{m}}(T)q^T$ with $a_{F_{m}}(T)\in \Q_p$, 
and a formal power series $G=\sum _{T\in \Lambda _{n}}a_{G}(T)q^T$
with $a_{G}(T)\in \Q_p$,  
we write 
\begin{align*}
\lim _{ m\to \infty} F_{m}=G\quad (p\text{-adic\ limit})
\end{align*}
if $v_p(F_{m}-G)\to \infty $ when $m\to \infty $.

We put 
\[{\boldsymbol X}=\Z_p\times \Z/(p-1)\Z. \]
We regard $\Z\subset {\boldsymbol X}$ via the embedding $n\mapsto (n,\widetilde{n})$ with $\widetilde{n}:=n$ mod $p-1$. 
Let $E_{k}^{n}=\sum _{T}a_{k}(T)q^T$ be the Siegel--Eisenstein series of degree $n$, weight $k$, level $1$ 
and the constant term $1$. 
For $(k,a)\in {\boldsymbol X}$, the $p$-adic Siegel--Eisenstein series $\widetilde{E}^{n}_{(k,a)}$
of degree $n$ and weight $(k,a)$ is defined (if convergent) as 
\[\widetilde{E}^{n}_{(k,a)}:=\lim _{m\to \infty }E_{k_m}^{n} \quad (p\text{-adic\ limit}), \]
where $k_m$ is a sequence of even integers such that $k_m\to \infty $ ($m\to \infty$) in the usual topology of $\mathbb{R}$ 
and $k_m\to (k,a)$ ($m\to \infty$) in ${\boldsymbol X}$. 
We write
\[\widetilde{E}^{n}_{(k,a)}=\sum _{T}\widetilde{a}_{(k,a)}(T)q^T. \]

\begin{Prop}
\label{Prop:p-adic_limit_FC}
Let $p$ be an odd prime. 
Let $\alpha \in \Z$ with $0\le \alpha \le p-2$ and
$k$ be a positive integer with $k\ge \frac{n+1}{2}$ such that $k+\alpha $ is even.
Let $\{k_m\}$ a sequence of even integers such that 
\[k_m\to (k,k+\alpha ) \in {\boldsymbol X},\quad k_m\to \infty \quad (m\to \infty ). \]
Then the Fourier coefficients of $\widetilde{E}^{n}_{(k,k+\alpha)}$ are given as follows.
For $T={\rm diag}(T_1,0)$ with $T_1\in \Lambda _{r}$ positive definite, we have 
\begin{align*}
&\widetilde{a}_{(k,k+\alpha )}(T)=2^{[(r+1)/2]}L^{(p)}(1-k,\omega ^\alpha )^{-1}\left(\prod _{j=1}^{[r/2]}L^{(p)}(1+2j-2k,\omega ^{2\alpha })^{-1}\right) \\
&~~~~~\times \prod _{q\neq p} F^r_q(T_1,\omega ^{\alpha }(q)q^{k-r-1})\times 
\begin{cases}
1 & \text{if}\ r\ \text{odd},\\
L^{(p)}(1+r/2-k,\chi _{T_1}\omega ^\alpha )  & \text{if}\ r\ \text{even}. 
\end{cases}
\end{align*}
Here $\omega $ is the Teichm\"uller character on $\Z_p$. 
\end{Prop}

\begin{proof}
We may suppose that $k_m=k+a(m)p^{b(m)}$ with $a(m)\equiv \alpha $ mod $(p-1)$, $b(m)\to \infty $ ($m\to \infty$). 
Let $T={\rm diag}(T_1,0)$ with $T_1\in \Lambda _{r}$ positive definite. 
Recall that   
\begin{align*}
&a_{k_m}(T)=2^{[(r+1)/2]}L(1-k_m,{\boldsymbol 1})^{-1}\left(\prod _{j=1}^{[r/2]}L(1+2i-2k_m,{\boldsymbol 1})^{-1}\right)\\
&~~~~~\times \prod _{q} F^r_q(T,q^{k_m-r-1})\times 
\begin{cases}
1 & \text{if}\ r\ \text{odd},\\
L(1+r/2-k,\chi _{T_1})  & \text{if}\ r\ \text{even}. 
\end{cases}
\end{align*}

Let $L_p(s,\varphi )$ be the Kubota--Leopoldt $p$-adic $L$-function for a primitive Dirichlet character $\varphi $. 
It is known that $L_p(s,\varphi )$ is a $p$-adic meromorphic (analytic if $\varphi \neq {\boldsymbol 1}$) function satisfying 
\begin{align}
\label{eq:Lp0}
L_p(1-s,\varphi )=-(1-\varphi \omega ^{-s}(p)p^{s-1})
\frac{B_{s,\varphi \omega ^{-s}}}{s}\qquad (1\le s\in \mathbb{Z}),  
\end{align}
where $B_{s,\chi }$ is the $s$-th generalized Bernoulli number with a character $\chi $. 
Note also that $L_p(s,{\boldsymbol 1})$ is analytic except for a pole at $s=1$. 
Then we have
\begin{align}
\label{eq:Lp}
L^{(p)}(1-s,\varphi \omega ^{-s})=L_p(1-s,\varphi) \qquad (1\le s\in \mathbb{Z}). 
\end{align}

First we show that 
\begin{align}
\label{eq:1}
\lim _{m\to \infty } L(1-k_m,{\boldsymbol 1})=L^{(p)}(1-k,\omega ^{\alpha }).
\end{align}
If we put $s=k_m$, $\varphi =\omega ^{k+\alpha }$ in (\ref{eq:Lp0}), then we have $\varphi \omega ^{-s}=\omega ^0={\boldsymbol 1}$ and therefore 
\begin{align*}
L_p(1-k_m,\omega ^{k+\alpha })=-(1-p^{k_m-1})\frac{B_{k_m}}{k_m}.
\end{align*}
Here $B_{k_m}=B_{k_m,\boldsymbol 1}$ is the $k_m$-th Bernoulli number. 
This implies 
\begin{align*}
L(1-k_m,{\boldsymbol 1})&=-\frac{B_{k_m}}{k_m}=\frac{L_p(1-k_m,\omega ^{k+\alpha })}{1-p^{k_m-1}}. 
\end{align*}
Taking a $p$-adic limit of the both sides, we obtain 
\begin{align*}
\lim _{m\to \infty }L(1-k_m,{\boldsymbol 1})&=L_p(1-k,\omega ^{k+\alpha })=L^{(p)}(1-k,\omega ^{\alpha }),  
\end{align*}
because of $k_m\to k$ ($m\to \infty$) $p$-adically. 
Here the last equality follows from (\ref{eq:Lp}). This proves (\ref{eq:1}). 
Note that $L^{(p)}(1-k,\omega ^{\alpha })\neq 0$ since $k+\alpha $ is even. 

In a similar way, we obtain 
\begin{align}
\nonumber 
&\lim _{m\to \infty } \prod _{j=1}^{[r/2]}L(1+2j-2k_m,{\boldsymbol 1})^{-1}=\prod _{j=1}^{[r/2]}L^{(p)}(1+2j-2k,\omega ^{2\alpha })^{-1}, \\
\label{eq:2}
&\lim _{m\to \infty }L_p(1+r/2-k_m,\chi _{T_1}\omega ^{k-r/2+\alpha })=L^{(p)}(1+r/2-k,\chi _{T_1}\omega ^\alpha )\\ \nonumber
&~~~~~~~~~~~~~~~~~~~~~~~~~~~~~~~~~~~~~~~~~~~~~~~(\text{for}\ r\ \text{even\ with} \ 0\le r \le n).  
\end{align}
Here, since $k\ge \frac{n+1}{2}$ implies $2j-2k<0$, the first equation above can be obtained using (\ref{eq:Lp}).  
Again, it is necessary to note that $L^{(p)}(1+2j-2k,\omega ^{2\alpha })\neq 0$ for each $j$ with $1\le j\le [r/2]$.

Next we show that 
\[\lim _{m\to \infty } \prod _{q} F^r_q(T_1,q^{k_m-r-1})=\prod _{q\neq p} F^r_q(T_1,\omega ^{\alpha }(q)q^{k-r-1}). \]
It suffices to show that 
\begin{align}
\label{eq:4}
&\lim_{m\to \infty }F^r_q(T_1,q^{k_m-r-1})=F^r_q(T_1,\omega ^{\alpha } (q)q^{k-r-1}), \\
\label{eq:5}
&\lim_{m\to \infty } F^r_p(T_1,p^{k_m-r-1})=1. 
\end{align}
Since 
\[q^{k_m-r-1}\equiv q^{a(m) p^{b(m)}}q^{k-r-1}\equiv \omega ^\alpha (q)q^{k-r-1}\bmod{p^{b(m)+1}}\]
and the constant term of $F^r_q(T_1,X)$ is $1$, 
we have 
\begin{align}
\label{eq:6}
&F^r_q(T_1,q^{k_m-r-1})\equiv F^r_q(T_1,\omega^\alpha (q)q^{k-r-1}) \bmod{p^{b(m)+1}},\\
\label{eq:7}
&F^r_p(T_1,p^{k_m-r-1})\equiv 1 \bmod{p^{b(m)+1}} \quad (\text{for}\ m\text{\ large}).
\end{align}
This shows (\ref{eq:4}), (\ref{eq:5}).

Finally, we need to check that this convergence is uniform in $T$.
The main issue is dependence on the quadratic character $\chi _{T_1}$.   
If the conductor of $\chi _{T_1}\omega ^{k-r/2+\alpha}$ is divisible by a prime $q$ other than $p$, 
then there exists a formal power series $\Phi_{T_1}(X)\in \Z_p[\![X]\!]$ such that 
$L_p(s,\chi _{T_1}\omega ^{k-r/2+\alpha })=\Phi_{T_1}((1+p)^s-1)$ (cf. Washington \cite[Theorem 7.10]{Wa}). 
By this fact and (\ref{eq:6}), (\ref{eq:7}), we see that the convergence is uniform in $T$. 
\end{proof}

\begin{Rem}
For the case of trivial character, one can avoid the use of Kubota--Leopoldt $p$-adic $L$-function 
(Kummer type congruences would be enough). 
\end{Rem}
\section{$p$-Adic limits of Siegel--Eisenstein series as classical modular forms}
\subsection{The case of quadratic character}
We identify the $p$-adic limit considered above with the (classical) Eisenstein series constructed in \S 2: 
\begin{Thm}
Let $p$ be an odd prime. 
Let $\delta =0$, $1$ and $k$ be an integer with $k\ge \frac{n+1}{2}$ such that $k+\frac{p-1}{2^{\delta }}$ is even. 
Let $\{k_m\}$ be a sequence of even integers such that 
\[k_m\to \left(k,k+\frac{p-1}{2^{\delta }}\right) \in {\boldsymbol X},\quad k_m\to \infty \quad (m\to \infty ). \]
Then we have 
\[\lim_{m\to \infty }E^{n}_{k_m}=\widetilde{E}^n_{(k,k+\frac{p-1}{2^\delta })}=
\begin{cases}
E^{n,(0)}_{k,{\boldsymbol 1}_p}(Z,k/2) & \text{if}\quad \delta =0,\\
E^{n,(0)}_{k,\chi_p}(Z,k/2) & \text{if}\quad \delta =1. 
\end{cases}
\quad (p\text{-adic\ limit})\] 
Here $E^{n,(0)}_{k,{\boldsymbol 1}_p}(Z,k/2)$ and $E^{n,(0)}_{k,\chi_p}(Z,k/2)$ are as in \S \ref{sec:Eisen}.
In particular, we have $\widetilde{E}^n_{(k,k)}\in M_{k}(\Gamma _0^{(n)}(p),{\boldsymbol 1}_p)$ and  
$\widetilde{E}^n_{(k,k+\frac{p-1}{2})}\in M_{k}(\Gamma _0^{(n)}(p),\chi _p)$. 
\end{Thm}
\begin{proof}
Keeping in mind that $\omega ^0={\boldsymbol 1}$, $\omega ^{\frac{p-1}{2}}=\chi _p$,
a comparison of the Fourier coefficients and confirms the coincidences. 
In fact, applying Proposition \ref{Prop:p-adic_limit_FC} as $\alpha =0$, $\frac{p-1}{2}$, 
we can confirm that $\widetilde{a}_{(k,k)}(T)=a(T)$ with $a(T)$ in Theorem \ref{thm: F-C trivial case} (1), 
and $\widetilde{a}_{(k,k+\frac{p-1}{2})}(T)=a(T)$ with $a(T)$ in Theorem \ref{thm: F-C nontrivial case} (1). 
\end{proof}

We compare the results of the papers \cite{Bo-Ki1,Bo-Ki3,Kat-Na}, concerning the conditions
on $p$, $k$ and $n$ necessary in the different approaches ($p$ always odd!): 

\begin{Rem}
\label{sec:comp}
The condition that $k+\frac{p-1}{2^{\delta}}$ should be even is imposed in all cases. 
Our first approach \cite{Bo-Ki1} by means of mod $p^m$ singular forms needs
$$ p>2k+1\quad \mbox{and}\quad p \,\, \mbox{regular} $$
In \cite{Bo-Ki3} we need
$$p>2k+1 \quad \mbox{and} \quad n\leq 2k+1. $$
In the present work, we need
$$k\ge \frac{n+1}{2}.$$
The motivation for considering the Siegel--Eisenstein series for $\Gamma_0^{(n)}(p)$ of weight $k$
and $U(p)$-eigenvalue 1 as a candidate comes from the first two approaches.

Finally we mention that the condition $p>2k+1$ does not appear in the work of Nagaoka \cite{Na} and Katsurada-Nagaoka \cite{Kat-Na} for the weights $1$ and $2$. 
\end{Rem} 

\subsection{The case of nonquadratic character}
We consider the case of nonquadratic character mod $p$. 
Let $\mu _{p-1}$ be the group of $(p-1)$-th root of unity in $\C^\times $ and $\chi :(\Z/p\Z)^\times \rightarrow \mu _{p-1}$ a nonquadratic character mod $p$. 
We fix an embedding $\sigma : \Q(\mu _{p-1})\hookrightarrow \Q_p$.  
For a positive integer $k$ with $k>n+1$,  
let $E^n_{k,\chi }=\sum _Ta_{k,\chi}(T)q^T$ be the Siegel--Eisenstein series of weight $k$ and degree $n$ with character $\chi $ which was studied by Takemori \cite{Ta}. 

 The following result is a generalization of Theorem 10 in Serre \cite{Se}. 
\begin{Thm}
Let $p$ be an odd prime. 
Let $k$ be a positive integer with $k>n+1$ and $\chi $ a nonquadratic character mod $p$ such that $\chi (-1)=(-1)^k$.  
We take $\alpha \in \Z$ with $1\le \alpha \le p-2$ ($\alpha \neq \frac{p-1}{2}$) such that $\chi ^\sigma =\omega ^\alpha $. 
Let $\{k_m\}$ be a sequence of even integers such that 
\[k_m\to (k,k+\alpha ) \in {\boldsymbol X},\quad k_m\to \infty \quad (m\to \infty ). \]
Then we have 
\[\lim_{m\to \infty }E^{n}_{k_m}=(E^{n}_{k,\chi })^\sigma. \] 
\end{Thm}
\begin{proof}
We can check easily that the coefficients $\widetilde{a}_{(k,k+\alpha )}(T)$ given in Proposition \ref{Prop:p-adic_limit_FC} coincide with the Fourier coefficients in \cite[Theorem 1.1]{Ta}. 
\end{proof}

\section*{Acknowledgment}
The second author was supported by JSPS KAKENHI Grant Number 24K06650, 
the third author was supported by JSPS KAKENHI Grant Number 22K03259.


\section*{Conflict of interest statement}
On behalf of all authors, the corresponding author states that there is no conflict of interest.

\section*{Data availability} 
During the work on this publication, no data sets were generated, used
or analyzed. Thus, there is no need for a link to a data repository.

\providecommand{\bysame}{\leavevmode\hbox to3em{\hrulefill}\thinspace}
\providecommand{\MR}{\relax\ifhmode\unskip\space\fi MR }
\providecommand{\MRhref}[2]{%
  \href{http://www.ams.org/mathscinet-getitem?mr=#1}{#2}
}
\providecommand{\href}[2]{#2}

\begin{flushleft}
Siegfried B\"ocherer\\
Kunzenhof 4B \\
79177 Freiburg, Germany \\
Email: boecherer@t-online.de
\end{flushleft}

\begin{flushleft}
Keiichi Gunji \\
Department of Mathematics\\
Chiba Institute of Technology\\
2-1-1 Shibazono, Narashino, Chiba 275-0023, Japan\\
Email: keiichi.gunji@it-chiba.ac.jp
\end{flushleft}

\begin{flushleft}
  Toshiyuki Kikuta\\
  Faculty of Information Engineering\\
  Department of Information and Systems Engineering\\
  Fukuoka Institute of Technology\\
  3-30-1 Wajiro-higashi, Higashi-ku, Fukuoka 811-0295, Japan\\
  E-mail: kikuta@fit.ac.jp
\end{flushleft}

\end{document}